\documentclass{article}
\usepackage[paperwidth=7in, paperheight=10in, margin=.875in]{geometry}

\usepackage{graphicx}%
\usepackage{multirow}%
\usepackage{amsmath,amssymb,amsfonts}%
\usepackage{amsthm}%
\usepackage{mathrsfs}%
\usepackage[title]{appendix}%
\usepackage{xcolor}%
\usepackage{textcomp}%
\usepackage{manyfoot}%
\usepackage{booktabs}%
\usepackage{algorithm}%
\usepackage{algorithmicx}%
\usepackage{algpseudocode}%
\usepackage{listings}%

\newtheorem{theorem}{Theorem}
\newtheorem{remark}{Remark}%

  \newtheorem{assumption}[subsection]{Assumption}
  \newtheorem{lemma}[theorem]{Lemma}
  \newtheorem{cor}[theorem]{Corollary}  
\begin{document}

\title{A consistent non-linear  Fokker-Planck model for a gas mixture of polyatomic molecules}


\author{Marlies Pirner}
\date{}

\maketitle

{\bf abstract:} We consider a multi component gas mixture with translational and internal energy degrees of freedom without chemical reactions assuming that the number of particles of each species remains constant. We will illustrate the derived model  in the case of two species, but the model can be  generalized to multiple species. The two species are allowed to have different degrees of freedom in internal energy and are modeled by a system of kinetic Fokker-Planck equations featuring two interaction terms to account for momentum and energy transfer between the species. 
We prove consistency of our model: conservation properties, positivity of the temperatures,  H-theorem and we characterize the equilibrium as two Maxwell distributions where all temperatures coincide. 
\section{Introduction}
 
 In this paper we shall concern ourselves with a kinetic description of gas mixtures for polyatomic molecules. In the case of mono atomic molecules and two species this is traditionally done via the Boltzmann equation or the Landau-Fokker-Planck equation for the density distributions $f_1$ and $f_2$, see for example \cite{Cercignani, Cercignani_1975, Chapman}. Under certain assumptions the complicated interaction terms of the Boltzmann or Landau-Fokker-Planck equation can be simplified by a Fokker-Planck approximation.
 This approximation should be constructed in a way such that it  has the same main properties of the Boltzmann or Landau-Fokker-Planck equation namely conservation of mass, momentum and energy, further it should have an H-theorem with its entropy inequality and the equilibrium must still be Maxwellian.  Fokker-Planck  models give rise to efficient numerical computations \cite{n3,n7,n8}. Evolution of a polyatomic gas mixture is very important in applications. But, most kinetic models deal with the case of a mono atomic  gas consisting of only one species.  \\
 
{\bf Literature on multi-species monoatomic Fokker Planck models:} The interest in multi-species Fokker-Planck models has been increased more and more recently. Models for gas mixtures of monoatomic molecules can be found in \cite{7, Hu, n7, n8,Agrawal, Pirner_Fokker, Cory}. The diffusion limit of a kinetic Fokker-Planck system for charged particles towards the Nernst-Planck equations was proved in \cite{15}. Furthermore, in \cite{7,11}, the limit of vanishing electron-ion mass ratios for non-homogeneous kinetic Fokker-Planck systems was investigated. In \cite{Hu}, the authors provide the first existence analysis of a multi-species Fokker-Planck system of the shape above. The works \cite{n7, n8} provide an extended Fokker-Planck model for hard-spheres gas mixtures to be able to also capture correct diffusion coefficients, mixture viscosity and heat conductivity coefficients in the hydrodynamic regime of the Navier-Stokes equations. \\

{\bf Literature on one species polyatomic Fokker-Planck models:}  In contrast to mono atomic molecules, in a polyatomic gas energy is not entirely stored in the kinetic energy of its molecules but also in their rotational and vibrational modes. One species Fokker-Planck models for polyatomic molecules can be found for example in \cite{Mathiaud1,Mathiaud2,Nagel,Morse}. Here, a additional dependency on the degrees of freedom in internal energy in the distribution function is introduced. There are to ansatzes in the literature to do so, a discrete dependency on the degrees of freedom in internal energy as it is done in \cite{Mathiaud1} or with a continuous dependency as in \cite{Nagel,Morse}. In \cite{Mathiaud2}, an extension is proposed to obtain the correct Prandtl number in the macroscopic regime of Navier-Stokes equations in the spirit of an ES-BGK extension of a Bathnagar-Gross-Krook (BGK) extension of the Boltzmann equation \cite{ES}.\\

 In this paper, we want to propose a Fokker-Planck model for a mixture of polyatomic molecules. As far as we can see this is not available in the literature at all. The construction of this model is motivated by the construction of the BGK model for a mixture of polyatomic molecules presented in \cite{Pirner_poly}. The advantage is the following. Kinetic models for polyatomic gases have two temperatures for the two different types of degrees of freedom, the translational and the internal energy degrees of freedom. Therefore, one expects two types of relaxations, a relaxation of the distribution function to a Maxwell distribution and a relaxation of the two temperatures to an equal value. The speed for the first type of relaxation may be faster or slower than the second type of relaxation. The construction in \cite{Pirner_poly} allows to separate these two types of relaxations. In addition, we use the idea of \cite{Pirner, Cory, Pirner_Fokker} and introduce free parameters in the model which can be used to fix exchange terms of momentum and energy, for more details see \cite{Pirner, Cory, Pirner_Fokker}.
 For simplification we present the model in the case of two species but the generalization to $N$ species is briefly described in remark \ref{rem4}. We do not consider chemical reactions. 
 We allow the two species to have different degrees of freedom in internal energy. For example, we may consider a mixture consisting of a mono atomic and a diatomic gas. \\

The outline of the paper is as follows: in section \ref{sec1} we will present a Fokker-Planck model for polyatomic molecules from \cite{Bernard} for two species of polyatomic molecules. 
In sections \ref{sec3} to \ref{sec6} we prove the conservation properties and the H-theorem. We show the positivity of all temperatures and quantify the structure of the equilibrium. 
In section \ref{sec9} we give an application in the case of a mixture consisting of a mono atomic and a polyatomic species.



\section{The space-homogeneous Fokker-Planck approximation for a mixture of polyatomic molecules}
\label{sec1}
\textcolor{black}{In this section we first want to motivate how our model with several coupled equations will look like. For the convenience of the reader, we will summarize all these equations again at the end of the section such that one sees the whole model at a glance.}
For simplicity in the following we consider a mixture composed of two different species. Let $x\in \mathbb{R}^d$ and $v\in \mathbb{R}^d, d \in \mathbb{N}$ be the phase space variables  and $t\geq 0$ the time. Let $M$ be the total number of different rotational and vibrational degrees of freedom and $l_k$ the number of internal degrees of freedom of species $k$, $k=1,2$. Note that the sum $l_1+l_2$ is not necessarily equal to $M$, because the two species could both have the same internal degree of freedom. Then $\eta \in \mathbb{R}^{M}$  is the variable for the internal energy degrees of freedom, $\eta_{l_k} \in \mathbb{R}^{M}$ coincides with $\eta$ in the components corresponding to the internal degrees of freedom of species $k$ and is zero in the other components. For example, we can consider two species both composed of molecules consisting of two atoms, such that the molecules have rotational degrees of freedom in addition to the three translational degrees of freedom. In general, a molecule consisting of two atoms has three possible axes around which it can rotate. But since the energy needed to rotate the molecule around the axes parallel to the line connecting the two atoms is very high (see for example \cite{Kelly}), this does not occur, so we have two rotational degrees of freedom. In this example we have $M=l_1=l_2=2$. \\ Since we want to describe two different species, our kinetic model has two distribution functions $f_1(x,v,\eta_{l_1},t)> 0$ and $f_2(x,v,\eta_{l_2},t) > 0$. 
 Furthermore we relate the distribution functions to  macroscopic quantities by mean-values of $f_k$, $k=1,2$ as follows
\begin{align}
\int f_k(v, \eta_{l_k}) \begin{pmatrix}
1 \\ v \\ \eta_{l_k} \\ m_k |v-u_k|^2 \\ m_k |\eta_{l_k} - \bar{\eta}_k |^2 
\end{pmatrix} 
dv d\eta_{l_k}=: \begin{pmatrix}
n_k \\ n_k u_k \\ n_k \bar{\eta}_k \\ d n_k T_k^{t} \\ l_k n_k T_k^{r} 
\end{pmatrix} , \quad k=1,2,
\label{moments}
\end{align} 
where $n_k$ is the number density, $u_k$ the mean velocity,  $T_k^{t}$ the mean temperature of the translation, and $T_k^{r}$ the mean temperature of the internal energy degrees of freedom for example rotation or vibration 
Note that in this paper we shall write $T_k^{t}$ and $T_k^{r}$ instead of $k_B T_k^{t}$ and $k_B T_k^{r}$, where $k_B$ is Boltzmann's constant. In  the following, we will require $\bar{\eta}_k=0$, which means that the energy in rotations clockwise is the same as in rotations counter clockwise. Similar for vibrations.

The distribution functions are determined by two equations to describe their time evolution. Furthermore we only consider binary interactions. 
So the particles of one species can interact with either themselves or with particles of the other species. In the model this is accounted for by introducing two interaction terms in both equations. These considerations allow us to write formally the system of equations for the evolution of the mixture. The following structure containing a sum of the collision operators is also given in \cite{Cercignani, Cercignani_1975,Chapman}. \\
We are interested in a Fokker-Planck approximation of the interaction terms. 
Then the model can be written as:

\begin{align} \begin{split} \label{BGK}
\partial_t f_1    &= c_{11} n_1 \left( \frac{\Lambda_{11}}{m_1} \nabla_{v} \cdot \left( M_1 \nabla_{v} \left( \frac{f_1}{M_1} \right) \right) + \frac{\Theta_{11}}{m_1} \nabla_{\eta_{l_1}} \cdot \left( M_1 \nabla_{\eta_{l_1}} \left( \frac{f_1}{M_{1}} \right) \right)  \right)
\\&+ c_{12} n_2 \left( \frac{\Lambda_{12}}{m_1} \nabla_{v} \cdot \left( M_{12} \nabla_{v} \left( \frac{f_1}{M_{12}} \right) \right)  + \frac{\Theta_{12}}{m_1} \nabla_{\eta_{l_1}} \cdot \left( M_{12} \nabla_{\eta_{l_1}} \left( \frac{f_1}{M_{12}} \right) \right)\right)\\ 
\partial_t f_2  &=c_{22} n_2 \left( \frac{\Lambda_{22}}{m_2} \nabla_{v} \cdot \left( M_2 \nabla_{v} \left( \frac{f_2}{M_2} \right) \right) + \frac{\Theta_{22}}{m_2} \nabla_{\eta_{l_2}} \cdot \left( M_2 \nabla_{\eta_{l_2}} \left( \frac{f_2}{M_{2}} \right) \right)  \right)
\\&+ c_{21} n_1 \left( \frac{\Lambda_{21}}{m_2} \nabla_{v} \cdot \left( M_{21} \nabla_{v} \left( \frac{f_2}{M_{21}} \right) \right)  + \frac{\Theta_{21}}{m_2} \nabla_{\eta_{l_2}} \cdot \left( M_{21} \nabla_{\eta_{l_2}} \left( \frac{f_2}{M_{21}} \right) \right)\right)
\end{split}
\end{align}
with the Maxwell distributions
\begin{align} 
\begin{split}
M_k(x,v,\eta_{l_k},t) &= \frac{n_k}{\sqrt{2 \pi \frac{\Lambda_k}{m_k}}^d } \frac{1}{\sqrt{2 \pi \frac{\Theta_k}{m_k}}^{l_k}} \exp \left({- \frac{|v-u_k|^2}{2 \frac{\Lambda_k}{m_k}}}- \frac{|\eta_{l_k}|^2}{2 \frac{\Theta_k}{m_k}}\right), 
\\
M_{kj}(x,v,\eta_{l_k},t) &= \frac{n_{kj}}{\sqrt{2 \pi \frac{\Lambda_{kj}}{m_k}}^d } \frac{1}{\sqrt{2 \pi \frac{\Theta_{kj}}{m_k}}^{l_k}} \exp \left({- \frac{|v-u_{kj}|^2}{2 \frac{\Lambda_{kj}}{m_k}}}- \frac{|\eta_{l_k}|^2}{2 \frac{\Theta_{kj}}{m_k}}\right), 
\end{split}
\label{BGKmix}
\end{align}
for $ j,k =1,2, j \neq k$, 
where $c_{11}$ and $c_{22}$ are friction constants related to intra-species collisions, while $c_{12}$ and $c_{21}$ are related to interspecies collisions. 
To be flexible in choosing the relationship between the friction constants, we now assume the relationship
\begin{equation} 
c_{12}=\varepsilon c_{21}. \quad 0 < \varepsilon \leq 1.
\label{coll}
\end{equation}
The restriction $ \varepsilon \leq 1$ is without loss of generality.
If $\varepsilon >1$, exchange the notation $1$ and $2$ and choose $\frac{1}{\varepsilon}.$ In addition, 
we assume that all friction constants are positive.

Since rotational/vibrational and translational degrees of freedom relax at a different rate, $T_k^{t}$ and $T_k^{r}$ will first relax to partial temperatures $\Lambda_k$ and $\Theta_k$ respectively. Conservation of internal energy then requires that at each time
\begin{align}
\frac{d}{2} n_k \Lambda_k = \frac{d}{2} n_k T_k^{t} +\frac{l_k}{2} n_k T_k^{r} - \frac{l_k}{2} n_k \Theta_k, \quad k=1,2. \label{internal}
\end{align}  
Thus, $\Lambda_k$ can be written as a function of $\Theta_k.$
In equilibrium we expect the two temperatures $\Lambda_k$ and $\Theta_k$ to coincide, so we close the system by adding the equations
 \begin{align}
 \begin{split}
 \partial_t M_k = 
 \frac{c_{kk} n_k}{Z_k^r} \frac{d+l_k}{d} \left( \frac{T_k}{m_k} \nabla_{v} \cdot \left( \widetilde{M}_k \nabla_v \left( \frac{M_k}{\widetilde{M}_k} \right) \right) +\frac{T_k}{m_k} \nabla_{\eta_{l_k}} \cdot \left( \widetilde{M}_k \nabla_{\eta_{l_k}} \left( \frac{M_k}{\widetilde{M}_k} \right) \right)\right) \\
+ c_{kj} n_j \left( \frac{T_{kj}}{m_k} \nabla_{v} \cdot \left( \widetilde{M}_{kj} \nabla_v \left( \frac{M_k}{\widetilde{M}_{kj}} \right) \right) 
+ \frac{T_{kj}}{m_k} \nabla_{\eta_{l_k}} \left( \widetilde{M}_{kj} \nabla_{\eta_{l_k}} \left( \frac{M_k}{\widetilde{M}_{kj}} \right) \right)\right) ,
 \end{split}
 \label{kin_Temp}
 \end{align}
for $j,k=1,2, j \neq k$, where $Z_k^r$ is a given parameter corresponding to the different rates of decays of translational and rotational/vibrational degrees of freedom. 
 Here $M_k$ is given by
\begin{align} 
M_k(x,v,\eta_{l_k},t) = \frac{n_k}{\sqrt{2 \pi \frac{\Lambda_k}{m_k}}^d } \frac{1}{\sqrt{2 \pi \frac{\Theta_k}{m_k}}^{l_k}} \exp({- \frac{|v-u_k|^2}{2 \frac{\Lambda_k}{m_k}}}- \frac{|\eta_{l_k}|^2}{2 \frac{\Theta_k}{m_k}}), \quad k=1,2,
\label{Maxwellian}
\end{align}
and $\widetilde{M}_k, \widetilde{M}_{kj}$ are given by 
\begin{align}
\widetilde{M}_k= \frac{n_k}{\sqrt{2 \pi \frac{T_k}{m_k}}^{d+l_k}} \exp \left(- \frac{m_k |v-u_k|^2}{2 T_k}- \frac{m_k|\eta_{l_k}|^2}{2 T_k} \right), \quad k=1,2, \\
\widetilde{M}_{kj}= \frac{n_k}{\sqrt{2 \pi \frac{T_{kj}}{m_k}}^{d+l_k}} \exp \left(- \frac{m_k |v-u_{kj}|^2}{2 T_{kj}}- \frac{m_k|\eta_{l_k}|^2}{2 T_{kj}} \right), \quad k=1,2.\label{Max_equ}
\end{align}
where $T_k$ is the total equilibrium temperature and is given by 
\begin{align}
T_k:= \frac{d \Lambda_k + l_k \Theta_k}{d+l_k}= \frac{d T^{t}_k + l_k T^{r}_k}{d+l_k}.
\label{equ_temp}
\end{align}
The second equality here follows from \eqref{internal}.
In addition, we define $T_{kj}$ as
\begin{align}
T_{kj}:= \frac{d \Lambda_{kj} + l_k \Theta_{kj}}{d+l_k}.
\label{equ_temp2}
\end{align}
\begin{remark}
Note, that we can write the derivatives in the collision operators in the equivalent form (see for example \cite{Cory})
\begin{align*}
 \frac{\Lambda_k}{m_k} \nabla_v \cdot \left(M_k \nabla_v \left( \frac{f_k}{M_k} \right) \right) &=  \left(\nabla_v \cdot \left( \left( v-u_k\right) f_k\right) + \nabla_v \cdot \left(\frac{\Lambda_k}{m_k} \nabla_v f_k\right)\right)
\\
 \frac{\Lambda_k}{m_k} \nabla_{\eta_{l_k}} \cdot \left(M_k \nabla_{\eta_{l_k}} \left( \frac{f_k}{M_k} \right) \right) &= \left(\nabla_{\eta_{l_k}} \cdot \left(  {\eta_{l_k}}  f_k\right) + \nabla_{\eta_{l_k}} \cdot \left(\frac{\Theta_k}{m_k} \nabla_{\eta_{l_k}} f_k\right)\right)
\\
 \frac{\Lambda_{kj}}{m_k} \nabla_v \cdot \left(M_{kj} \nabla_v \left( \frac{f_k}{M_{kj}} \right) \right) &= \left(\nabla_v \cdot \left( \left( v-u_{kj}\right) f_k\right) + \nabla_v \cdot \left(\frac{\Lambda_{kj}}{m_k} \nabla_v f_k\right)\right)
\\
\frac{\Lambda_{kj}}{m_k} \nabla_{\eta_{l_k}} \cdot \left(M_{kj} \nabla_{\eta_{l_k}} \left( \frac{f_k}{M_{kj}} \right) \right) &=  \left(\nabla_{\eta_{l_k}} \cdot \left( {\eta_{l_k}}  f_k\right) + \nabla_{\eta_{l_k}} \cdot \left(\frac{\Theta_{kj}}{m_k} \nabla_{\eta_{l_k}}f_k\right)\right)
\\
 \frac{T_k}{m_k} \nabla_v \cdot \left(\widetilde{M}_k \nabla_v \left( \frac{M_k}{\widetilde{M}_k} \right) \right) &=  \left(\nabla_v \cdot \left( \left( v-u_k\right) M_k\right) + \nabla_v \cdot \left(\frac{T_k}{m_k} \nabla_v M_k\right)\right)
\\
 \frac{T_k}{m_k} \nabla_{\eta_{l_k}} \cdot \left(\widetilde{M}_k \nabla_{\eta_{l_k}} \left( \frac{M_k}{\widetilde{M}_k} \right) \right) &= \left(\nabla_{\eta_{l_k}} \cdot \left( \eta_{l_k} M_k\right) + \nabla_{\eta_{l_k}} \cdot \left(\frac{T_k}{m_k} \nabla_{\eta_{l_k}} M_k\right)\right)
\\
\frac{T_{kj}}{m_k} \nabla_v \cdot \left(\widetilde{M}_{kj} \nabla_v \left( \frac{M_k}{\widetilde{M}_{kj}} \right) \right) &=  \left(\nabla_v \cdot \left( \left( v-u_{kj}\right) M_k\right) + \nabla_v \cdot \left(\frac{T_{kj}}{m_k} \nabla_v M_k\right)\right)
\\
\frac{T_{kj}}{m_k} \nabla_{\eta_{l_k}} \cdot \left(\widetilde{M}_{kj} \nabla_{\eta_{l_k}} \left( \frac{M_k}{\widetilde{M}_{kj}} \right) \right) &=  \left(\nabla_{\eta_{l_k}} \cdot \left( {\eta_{l_k}} M_k\right) + \nabla_{\eta_{l_k}} \cdot \left(\frac{T_{kj}}{m_k} \nabla_{\eta_{l_k}} M_k\right)\right)
\end{align*}
\end{remark}

If we multiply \eqref{kin_Temp} by $|\eta_{l_k}|^2$, integrate with respect to $v$ and $\eta_{l_k}$ and use \eqref{equ_temp}, we obtain  
\begin{align}
\begin{split}
\partial_t(n_k \Theta_k) +   \nabla_x\cdot (n_k \Theta_k u_k) = \frac{c_{kk} n_k}{Z_k^r} n_k (\Lambda_k - \Theta_k)&+ c_{kj} n_j n_k(T_{kj} - \Theta_k) ,\quad k=1,2.
\end{split}
\label{relax}
\end{align}

We see that in this model the term $\frac{c_{kk} n_k}{Z_k^r} n_k (\Lambda_k - \Theta_k)$ describes the relaxation of the two temperatures to a common temperature.

In addition,  \eqref{BGK} and \eqref{kin_Temp} are consistent. If we multiply the equations for species $k$ of \eqref{BGK} and \eqref{kin_Temp}  by $v$ and integrate with respect to $v$ and $\eta_{l_k}$, we get in both cases for the right-hand side 
$$ c_{kj} n_j n_k  (u_{jk} - u_k),$$ and if we compute the total internal energy of both equations, we obtain in both cases $$ \frac{1}{2} c_{kj} n_k n_j [d \Lambda_{jk} + l_j \Theta_{jk} - ( d \Lambda_j + l_j \Theta_j)].$$
The motivation of choosing $T_{kj}$ was to guarantee the last one, namely to ensure that \eqref{BGK} and \eqref{kin_Temp} are consistent in the momentum exchange and internal energy exchange.
 \\
 
With this choice of the Maxwell distributions $M_1$ and $M_2$ have the same densities, mean velocities and internal energies as $f_1$ respective $f_2$. This guarantees the conservation of mass, momentum and energy in interactions of one species with itself. 
The remaining parameters $n_{12}, n_{21}, u_{12}, u_{21}, \Lambda_{12}$ , $\Lambda_{21}$, $\Theta_{12}$ and $\Theta_{21}$ will be determined further down using conservation of the number of particles, total momentum and total energy, together with some symmetry considerations.
\textcolor{black}{We will determine $n_{12}$ and $n_{21}$ in equation \eqref{density} using conservation of the number of particles. The velocities $u_{12}$ and $u_{21}$ will be determined in equations \eqref{convexvel} and \eqref{veloc} by using conservation of total momentum. Last, the parameters  $\Lambda_{12}$ , $\Lambda_{21}$, $\Theta_{12}$ and $\Theta_{21}$ will be determined in theorem \ref{consenergy} and remark \ref{detpar}.  }
\\ \\
\textcolor{black}{Now, for the convenience of the reader, we want to write down our model again that one sees on the first view which equations we want to couple. Our Fokker-Planck model for two species coupled with one relaxation equation and one algebraic equation for the temperatures for each species can be written as
\begin{align*} \begin{split} 
\partial_t f_1   &= c_{11} n_1 \left( \frac{\Lambda_{11}}{m_1} \nabla_{v} \cdot \left( M_1 \nabla_{v} \left( \frac{f_1}{M_1} \right) \right) + \frac{\Theta_{11}}{m_1} \nabla_{\eta_{l_1}} \cdot \left( M_1 \nabla_{\eta_{l_1}} \left( \frac{f_1}{M_{1}} \right) \right)  \right)
\\&+ c_{12} n_2 \left( \frac{\Lambda_{12}}{m_1} \nabla_{v} \cdot \left( M_{12} \nabla_{v} \left( \frac{f_1}{M_{12}} \right) \right)  + \frac{\Theta_{12}}{m_1} \nabla_{\eta_{l_1}} \cdot \left( M_{12} \nabla_{\eta_{l_1}} \left( \frac{f_1}{M_{12}} \right) \right)\right)\\ 
\partial_t f_2  &=c_{22} n_2 \left( \frac{\Lambda_{22}}{m_2} \nabla_{v} \cdot \left( M_2 \nabla_{v} \left( \frac{f_2}{M_2} \right) \right) + \frac{\Theta_{22}}{m_2} \nabla_{\eta_{l_2}} \cdot \left( M_2 \nabla_{\eta_{l_2}} \left( \frac{f_2}{M_{2}} \right) \right)  \right)
\\&+ c_{21} n_1 \left( \frac{\Lambda_{21}}{m_2} \nabla_{v} \cdot \left( M_{21} \nabla_{v} \left( \frac{f_2}{M_{21}} \right) \right)  + \frac{\Theta_{21}}{m_2} \nabla_{\eta_{l_2}} \cdot \left( M_{21} \nabla_{\eta_{l_2}} \left( \frac{f_2}{M_{21}} \right) \right)\right)
\end{split}
\end{align*}
\begin{align*}
\frac{d}{2} n_k \Lambda_k = \frac{d}{2} n_k T_k^{t} +\frac{l_k}{2} n_k T_k^{r} - \frac{l_k}{2} n_k \Theta_k,  
\end{align*}
\begin{align*}
\begin{split}
\partial_t(n_k \Theta_k)  = \frac{c_{kk} n_k}{Z_k^r} n_k (\Lambda_k - \Theta_k)+ c_{kj} n_j n_k(T_{kj} - \Theta_k^r) ,\quad k=1,2.
\end{split}
\end{align*}}
\section{Conservation properties}
\label{sec3}
Conservation of the number of particles and total momentum of the model for mixtures described in \textcolor{black}{section \ref{sec1}} are shown in the same way as in the case of mono atomic molecules. 
Conservation of the number of particles and of total momentum are guaranteed by the following choice of the mixture parameters:\\ \\
If we assume that \begin{align} n_{12}=n_1 \quad \text{and} \quad n_{21}=n_2,  
\label{density} 
\end{align}
we have conservation of the number of particles, see Theorem 1 in \cite{Pirner_Fokker}.
If we further assume that $u_{12}$ is a linear combination of $u_1$ and $u_2$
 \begin{align}
u_{12}= \delta u_1 + (1- \delta) u_2, \quad \delta \in \mathbb{R},
\label{convexvel}
\end{align} then we have conservation of total momentum
provided that
\begin{align}
u_{21}=u_2 - \frac{m_1}{m_2} \varepsilon (1- \delta ) (u_2 - u_1),
\label{veloc}
\end{align}
see Theorem 2 in \cite{Pirner_Fokker}.

In the case of total energy we have a difference for the polyatomic case compared to the monoatomic one. So we explicitly consider this in the following theorem.
\begin{theorem}[Conservation of total energy]
Assume \eqref{coll}, conditions \eqref{density}, \eqref{convexvel} and \eqref{veloc} and assume that $\Lambda_{12}$ and $\Theta_{12}$ satisfy
\begin{align}
\begin{split}
d \Lambda_{12} +l_1 \Theta_{12} = \alpha (d \Lambda_1 +l_1 \Theta_1) + (1-\alpha) (d \Lambda_2 +l_2 \Theta_2) + \gamma |u_1-u_2|^2 \label{contemp}
\end{split}
\end{align}
Then we have conservation of total energy
\begin{align*}
\int \frac{m_1}{2} (|v|^2 + |\eta_{l_1}|^2 ) Q_{12}(f_1,f_2) dv d\eta_{l_1} +
\int \frac{m_2}{2} (|v|^2 + |\eta_{l_2}|^2 ) Q_{21}(f_2,f_1) dv d\eta_{l_2}= 0,
\end{align*}
\label{consenergy}
\end{theorem}
provided that
\begin{align}
\begin{split}
d \Lambda_{21} + l_2 \Theta_{21}=\left[  \varepsilon m_1 (1- \delta)  - \varepsilon \gamma \right] |u_1 - u_2|^2 + (1- \varepsilon (1-\alpha)) (d \Lambda_2 + l_2 \Theta_2)  \\+ \varepsilon (1- \alpha) (d \Lambda_1 + l_1 \Theta_1) 
\label{temp}
\end{split}
\end{align}
\begin{proof}
Using the definition of the energy exchange of species $1$ and equation \eqref{internal}, we obtain
\begin{align*}
F_{E_{1,2}}:&= \int \frac{m_1}{2} (|v|^2 + |\eta_{l_1}|^2)  Q_{12}(f_1,f_2) dv d\eta_{l_1} 
\\&= c_{12} m_1 n_1 n_2 u_1 \cdot (u_{12}-u_1)+ c_{12}  n_1 n_2 d (\Lambda_{12} - T_1^{t})  +l_1 c_{12} n_1 n_2 ( \Theta_{12} - T_1^{r} ) \\&= c_{12} m_1 n_1 n_2 u_1 \cdot (u_{12}-u_1)+ d  c_{12} n_1 n_2  (\Lambda_{12} - \Lambda_1)  +l_1 c_{12} n_1 n_2 ( \Theta_{12} - \Theta_1 ).
\end{align*}
 Next, we will insert the definitions of $u_{12}$ and $d \Lambda_{12}+l_1 \Theta_{12}$  given by \eqref{convexvel} and \eqref{contemp}. Analogously the energy exchange of species $2$ towards $1$ is 
$$
F_{E_{2,1}}=  c_{21} m_2 n_1 n_2 u_2 \cdot (u_{21}-u_2)+ d  c_{21} n_1 n_2  (\Lambda_{21} - \Lambda_2)  +l_2  c_{21} n_1 n_2 ( \Theta_{21} - \Theta_2 ).
$$
Substitute $u_{21}$ with \eqref{veloc}. 
This permits to rewrite the energy exchange as 
\begin{align}
\begin{split}
F_{E_{1,2}}&= \varepsilon c_{21} n_2 n_1 m_1 (1-\delta) u_1 \cdot (u_2-u_1) \\ &+  \varepsilon c_{21} n_1 n_2  \left[((1-\alpha)  \left( l_2 \Theta_2 +d \Lambda_2 - (d \Lambda_1 +l_1 \Theta_1)\right)  + \gamma   |u_1-u_2|^2\right],
\end{split}
 \label{flux_en_12}
\end{align}
\begin{align}
\begin{split}
F_{E_{2,1}} = &c_{21} m_2 n_1 n_2  (1-\delta) \frac{m_1}{m_2} \varepsilon u_2 \cdot (u_1-u_2)+ \frac{1}{2} c_{21} n_1 n_2 \big[ \varepsilon ( 1- \alpha) d( \Lambda_1 - \Lambda_2)\\&+ d  c_{21} n_1 n_2  (\Lambda_{21} - \Lambda_2)  +l_2  c_{21} n_1 n_2 ( \Theta_{21} - \Theta_2 ).
\end{split}
\label{flux_en_21}
\end{align}
Adding these two terms, we see that the total energy is conserved provided that $d\Lambda_{21} + l_2\Theta_{21}$ are given by \eqref{temp}.
\end{proof}
\begin{remark}
The energy flux between the two species is zero if and only if $u_1=u_2,$ $d \Lambda_1+l_1 \Theta_1=d\Lambda_2+ l_2 \Theta_2$ provided that $\alpha, \delta <1$ and $\gamma >0$.
\end{remark}
\begin{remark}
In order to ensure conservation of total energy, we do not need to assume an explicit expression on $\Lambda_{12}$ and $\Theta_{12}$, only on $d \Lambda_{12} +l_1 \Theta_{12}$. In addition,  we get only one condition on $d \Lambda_{21} + l_2 \Theta_{21}$ given by \eqref{temp}, but not an explicit formula for $\Lambda_{21}$ and $\Theta_{21}$. Later, for the H-Theorem, we will need explicit formulas for $\Lambda_{12}, \Theta_{12}, \Lambda_{21}, \Theta_{21}$, separately. We will assume  $ \Lambda_{12}=  \Theta_{12}$ and $ \Lambda_{21} =  \Theta_{21}$.
\label{detpar}
\end{remark}
\begin{remark}
\label{rem4}
The fact that we only consider the two species case is just for simplicity. We can also extend the model to more than two species, because we assume that we only have binary interactions. So if we consider collision terms given by \begin{align*} Q_i(&f_1,...,f_N):= \sum_{j=1}^N Q_{ij}(f_i,f_j) \\&= c_{ii} n_i \left( \frac{\Lambda_{ii}}{m_i} \nabla_{v} \cdot \left( M_i \nabla_{v} \left( \frac{f_i}{M_i} \right) \right) + \frac{\Theta_{ii}}{m_i} \nabla_{\eta_{l_i}} \cdot \left( M_i \nabla_{\eta_{l_i}} \left( \frac{f_i}{M_{i}} \right) \right)  \right)
\\&+ \sum_{j=1, j\neq i}^N \left( c_{ij} n_j \left( \frac{\Lambda_{ij}}{m_i} \nabla_{v} \cdot \left( M_{ij} \nabla_{v} \left( \frac{f_i}{M_{ij}} \right) \right)  + \frac{\Theta_{ij}}{m_i} \nabla_{\eta_{l_i}} \cdot \left( M_{ij} \nabla_{\eta_{l_i}} \left( \frac{f_i}{M_{ij}} \right) \right)\right)\right)
\end{align*}
for $i=1,...N$, we expect that we have conservation of total momentum and total energy in every interaction of species $i$ with species $j$. This means we require
$$ \int \begin{pmatrix}
v \\ v^2
\end{pmatrix} Q_{ij}(f_i,f_j) dv + \int \begin{pmatrix}
v \\ v^2
\end{pmatrix} Q_{ji}(f_j,f_i) dv =0,$$ for every $i,j=1,...N, ~ i \neq j$ and so it reduces to the two species case. 
\end{remark}

\section{Positivity of the internal energies}
\label{sec4}
\begin{theorem}
Assume that $f_1(x,v, \eta_{l_1},t), f_2(x,v,\eta_{l_2},t) > 0$. Then  all temperatures $\Lambda_1$, $\Lambda_2$, $\Theta_1$, $\Theta_2$, the internal energies $d \Lambda_{12} + l_1 \Theta_{12}$ given by \eqref{contemp},  and $d \Lambda_{21} + l_2 \Theta_{21}$ determined by \eqref{temp} are positive provided that 
 \begin{align}
0 \leq \gamma  \leq m_1 (1-\delta)
 \label{gamma}
 \end{align}
\end{theorem}
\begin{proof}
The temperatures $\Lambda_1, \Lambda_2, \Theta_1, \Theta_2$ and the internal energy $d \Lambda_{12} + l_1\Theta_{12}$  are positive by definition because they are integrals or convex combinations of positive functions. So the only thing to check is when the internal energy $d \Lambda_{21} + l_1 \Theta_{21}$ in \eqref{temp} is positive. The resulting condition is given by \eqref{gamma}.
\end{proof}
\begin{remark}
Since $\gamma \geq 0$ is a non-negative number, so the right-hand side of the inequality in \eqref{gamma} must be non-negative. This condition is equivalent to 
\begin{align}
  \delta \leq 1.
\label{gammapos}
\end{align}
\end{remark}
\section{H-Theorem}
\label{sec6}
In this section we will prove that our model admits an entropy with an entropy inequality.  For this, we make the following additional assumptions. 
\begin{assumption}
\label{assump}
\hspace{1cm}
\begin{itemize}
\item  The friction constants $c_{11}, c_{22}, c_{12}, c_{21}$ are assumed to be independent of $t$.
\item  We assume that $\Lambda_{12}=\Theta_{12}$ and $\Lambda_{21}=\Theta_{21}$.  Then, we have separate expressions which determine $\Lambda_{12}, \Theta_{12}, \Lambda_{21}$ and $\Theta_{12}$ given by \eqref{internal}.
\item The initial data of $\Lambda_k$ and $\Theta_k$ are chosen such that $\Theta_k(0)-\Lambda_k(0)$ has the same sign as $T_k^r(0)- T_k^t(0)$. 
\item We make the following stronger assumptions on the free parameters
\begin{align*}
\gamma= \frac{\varepsilon}{1+ \varepsilon} m_1 (1-\delta), \quad  \delta\geq \max \{ \frac{1}{1+\varepsilon}, \frac{1+\varepsilon(1-\frac{m_1}{m_2})}{1+\varepsilon} \}, \quad \alpha \geq \varepsilon \frac{l_1+d}{2d +l_1+l_2}
\end{align*}
Note, that a typical choice in the literature for $\varepsilon$ is $\varepsilon=\frac{m_2}{m_1}$. Then the two terms in the maximum coincide. 
\item We assume the parameters $Z^r_1,Z^r_2$ determining the rate of $\Lambda$ and $\Theta$ to a common value satisfy $\frac{Z^r_2}{Z^r_1}=\frac{d+l_1}{d+l_2}$.
\end{itemize}
\end{assumption}
We start with the terms related to intra-species collisions.

\begin{lemma}[Contribution to the H-theorem from the one species relaxation terms]
Assume $f_1, f_2 >0$. 
Then
\begin{align*} 
I_k:&= \int \int  \frac{\Lambda_{kk}}{m_k} \nabla_{v} \cdot \left( M_k \nabla_{v} \left( \frac{f_k}{M_k} \right) \right) \ln f_k dv d\eta_{l_k} 
\\ &+ \int \int \frac{\Theta_{kk}}{m_k} \nabla_{\eta_{l_k}} \cdot \left( M_k \nabla_{\eta_{l_k}} \left( \frac{f_k}{M_{k}} \right) \right) \ln f_k dv d\eta_{l_k} \\&+ 
\int \int \frac{T_k}{m_k} \nabla_{v} \cdot \left( \widetilde{M}_k \nabla_v \left( \frac{M_k}{\widetilde{M}_k}  \right) \right)\ln M_k dv d\eta_{l_k} \\&+ \int \int \frac{T_k}{m_k} \nabla_{\eta_{l_k}} \cdot \left( \widetilde{M}_k \nabla_{\eta_{l_k}} \left( \frac{M_k}{\widetilde{M}_k} \right) \right) \ln M_k dv d\eta_{l_k} \leq 0, 
 \end{align*}
for  $k=1,2,$ with equality if and only if $f_k=M_k$ and  $\Lambda_k=\Theta_k=T_k^{r} = T_k^{t}$.
\label{one_species}
\end{lemma}
\begin{proof}
The term $I_k$ can be written in the equivalent form
\begin{align*} 
I_k = \int \int  \left( \nabla_v   \cdot ((v-u_k)f_k) + \nabla_v \cdot \left( \frac{\Lambda_k}{m_k} \nabla_v f_k \right) \right)\ln f_k dv d\eta_{l_k}
\\ +  \int \int  \left( \nabla_{\eta_{l_k}}   \cdot (\eta_{l_k} f_k) + \nabla_{\eta_{l_k}} \cdot \left( \frac{\Theta_k}{m_k} \nabla_{\eta_{l_k}} f_k \right) \right)\ln f_k dv d\eta_{l_k}
\\+ \int \int  \left( \nabla_v   \cdot ((v-u_k) M_k) + \nabla_v \cdot \left( \frac{T_k}{m_k} \nabla_v M_k \right) \right)\ln M_k dv d\eta_{l_k}
\\ +  \int \int  \left( \nabla_{\eta_{l_k}}   \cdot (\eta_{l_k} M_k) + \nabla_{\eta_{l_k}} \cdot \left( \frac{T_k}{m_k} \nabla_{\eta_{l_k}} M_k \right) \right)\ln M_k dv d\eta_{l_k}, \end{align*}
see for example section 2 in \cite{Cory}.
After integration by parts, we obtain
\begin{align*} 
I_k &=-  \int \int  \sum_{i=1}^d \left( (v_i-u_{k,i}) \partial_{v_i} f_k + \frac{\Lambda_k}{m_k}  \frac{(\partial_{v_i} f_k)^2}{f_k}  \right) dv d\eta_{l_k}
\\ &-  \int \int  \sum_{i=1}^d \left( \eta_{l_{k,i}} \partial_{\eta_{l_{k,i}}} f_k + \frac{\Theta_k}{m_k}  \frac{(\partial_{\eta_{\eta_{l_k,i}}} f_k)^2}{f_k}  \right)  dv d\eta_{l_k}
\\&- \int \int  \sum_{i=1}^d \left( (v_i-u_{k,i}) \partial_{v_i} M_k + \frac{T_k}{m_k}  \frac{(\partial_{v_i} M_k)^2}{M_k}  \right)  dv d\eta_{l_k}
\\ &-  \int \int   \sum_{i=1}^d \left( \eta_{l_{k,i}} \partial_{\eta_{l_{k,i}}} M_k + \frac{\Theta_k}{m_k}  \frac{(\partial_{\eta_{\eta_{l_k,i}}} M_k)^2}{M_k}  \right)  dv d\eta_{l_k}
\\&=: I_{k,1}+I_{k,2}+ I_{k,3}+I_{k,4}
 \end{align*}
 We continue with writing $I_{k,1}$ as
 \begin{align*}
I_{k,1} &=  -  \int \int  \sum_{i=1}^d \left( (v_i-u_{k,i})  f_k + \frac{\Lambda_k}{m_k}  \partial_{v_i} f_k \right) \left(\frac{\partial_{v_i} f_k}{f_k} + \frac{v_i -u_{k,i}}{\Lambda_{k,i}/m_k} - \frac{v_i -u_{k,i}}{\Lambda_{k,i}/m_k} \right) dv d\eta_{l_k}
\\ &= -  \int \int  \sum_{i=1}^d \frac{\Lambda_k}{f_k} \left(\frac{ (v_i-u_{k,i}) }{\Lambda_k/m_k} f_k + \partial_{v_i} f_k  \right) ^2 dv d\eta_{l_k} \\&+  \int \int  \sum_{i=1}^d  \left(\frac{ (v_i-u_{k,i})^2 }{\Lambda_k/m_k} f_k + (v_i-u_{k,i}) \partial_{v_i} f_k  \right)  dv d\eta_{l_k} \\ &\leq 0 + d n_k \left( \frac{T_k^t}{\Lambda_k} -1 \right)
\end{align*}
In a similar way, one can show 
\begin{align*}
I_{k,2} \leq l_k n_k \left(\frac{T_k^r}{\Theta_k}-1 \right)
\end{align*}
It remains to estimate $I_{k,3}+I_{k,4}$. Since in this terms we have Maxwellians, we can directly compute the integrals and obtain
\begin{align*}
I_{k,3}+I_{k,4} = d n_k (1- \frac{T_k}{\Lambda_k} ) + l_k n_k (1- \frac{T_k}{\Theta_k})
\end{align*}
If we sum all the terms up, we obtain
\begin{align}
\begin{split}
I_k = d n_k (\frac{T_k^t}{\Lambda_k} - \frac{T_k}{\Lambda_k}) + l_k n_k ( \frac{T_k^r}{\Theta_k} - \frac{T_k}{\Theta_k}) = d n_k \frac{l_k}{d+l_k} \frac{T_k^t-T_k^r}{\Lambda_k} + l_k n_k \frac{d}{d+l_k} \frac{T_k^r-T_k^t}{\Theta_k} \\= n_k \frac{d~ l_k}{d+l_k} (T_k^t -T_k^r)( \frac{1}{\Lambda_k} - \frac{1}{\Theta_k}) = n_k \frac{d ~l_k}{d+l_k} (T_k^t -T_k^r)( \frac{\Theta_k-\Lambda_k}{\Lambda_k \Theta_k})
\end{split}
\label{I_k}
\end{align}
We conclude with proving that in the space-homogeneous case $T_k^t - T_k^r$ has the same sign as $\Theta_k-\Lambda_k$ if we require this initially. We take the equation for $\Theta_k$ given by \eqref{relax} and use that $n_k$ is constant in the space-homogeneous case (due to mass conservation) and  obtain \begin{align*}
\partial_t \Theta_k = \frac{c_{kk} n_k}{Z_k^r} n_k (T_k-\Theta_k) + c_{kj} n_j n_k (T_{kj} -\Theta_k)
\end{align*}
If we multiply \eqref{kin_Temp} by $|v-u_k|^2$ and integrate, we obtain
\begin{align*}
\partial_t \Lambda_k = \frac{c_{kk} n_k}{Z_k^r} n_k (T_k-\Lambda_k) + c_{kj} n_j n_k (T_{kj} -\Lambda_k)
\end{align*}
If we subtract, the first equation from the second one, we obtain
\begin{align*}
\partial_t (\Theta_k-\Lambda_k) = \frac{c_{kk} n_k}{Z^k_r} n_k (\Lambda_k -\Theta_k) +c_{kj} n_j n_k (\Lambda_k -\Theta_k)
\end{align*}
This leads to
\begin{align}
\Theta_k-\Lambda_k =e^{- (\frac{c_{kk} n_k}{Z^k_r} n_k + c_{kj} n_j n_k)t} (\Theta_k(0) -\Lambda_k(0))
\label{diffLT}
\end{align}
We observe that $\Theta_k-\Lambda_k$ remains the same sign as it is initially for all $t>0$.
Similar, from \eqref{BGK}, we can derive
\begin{align*}
\partial_t T_k^r = c_{kk} n_k n_k (\Theta_k -T_k^r) + c_{kj} n_j n_k (\Theta_{kj} - T_k^r) \\ \partial_t T_k^t = c_{kk} n_k n_k (\Lambda_k -T_k^t) + c_{kj} n_j n_k (\Lambda_{kj} - T_k^t)\end{align*}
With the second assumption of assumption \ref{assump}, we get
\begin{align*}
\partial_t(T_k^r - T_k^t) = c_{kk} n_k n_k (\Theta_k - \Lambda_k + T_k^t - T_k^r ) + c_{kj} n_j n_k (T_k^t- T_k^r)
\end{align*}
We insert \eqref{diffLT} and get
\begin{align*}
\partial_t(T_k^r - T_k^t) = c_{kk} n_k n_k (e^{-( \frac{c_{kk} n_k}{Z^k_r} n_k +c_{kj} n_k n_j)t} (\Theta_k(0) -\Lambda_k(0)) + T_k^t - T_k^r ) \\+ c_{kj} n_j n_k (T_k^t- T_k^r)
\end{align*}
With Duhamels formula, we obtain an equation of the form
\begin{align*}
T_k^r- T_k^t = e^{-(c_{kk} n_k + c_{kj} n_j) n_k t} (T_k^r(0) - T_k^t(0)) + \int_0^t c_{kk} n_k n_k C(t,s) (\Theta(0)- \Lambda(0)) ds, 
\end{align*}
with a positive constant $C(t,s).$ Here, we can see, that $T_k^r-T_k^t$ remains its sign if we choose the sign of $\Theta_k(0)- \Lambda_k(0)$ equal to the sign of  $T_k^r(0)-T_k^t(0)$. This leads to a non-positive sign of \eqref{I_k}.

\end{proof}
In the following we will consider the H-Theorem for the mixture interaction terms.
\textcolor{black}{Define 
$\frac{1}{z} := \frac{1}{Z_k^r} \frac{d+l_k}{d}, ~ k=1,2$ and the
 total entropy \begin{align}
 H(f_1,f_2) = \int (f_1 \ln f_1  +  z M_1 \ln M_1) dv d\eta_{l_1} + \int (f_2 \ln f_2+ z M_2 \ln M_2 ) dvd\eta_{l_2}.
 \label{H-funct}
 \end{align}} Note, that $z$ is independent of $k$ due to the last assumption in assumptions \eqref{assump}.

\begin{theorem}[H-theorem for mixture]
Assume $f_1, f_2 >0$. 
Assume the relationship between the collision frequencies \eqref{coll}, the conditions for the interspecies Maxwellians \eqref{density}, \eqref{convexvel}, \eqref{veloc}, \eqref{contemp} and \eqref{temp}, the positivity of all temperatures and assumptions \ref{assump}, then
\begin{align*}
I= \int \int  \frac{\Lambda_{12}}{m_1} \nabla_{v} \cdot \left( M_{12} \nabla_{v} \left( \frac{f_1}{M_{12}} \right) \right)\ln f_1 dv d\eta_{l_1} 
\\ + \int \int \frac{\Theta_{12}}{m_1} \nabla_{\eta_{l_1}} \cdot \left( M_{12} \nabla_{\eta_{l_1}} \left( \frac{f_1}{M_{12}} \right) \right) \ln f_1 dv d\eta_{l_1} \\+ 
 \int \int  \frac{\Lambda_{21}}{m_2} \nabla_{v} \cdot \left( M_{21} \nabla_{v} \left( \frac{f_2}{M_{21}} \right) \right) \ln f_2 dv d\eta_{l_1} 
\\ + \int \int \frac{\Theta_{21}}{m_2} \nabla_{\eta_{l_2}} \cdot \left( M_{21} \nabla_{\eta_{l_2}} \left( \frac{f_2}{M_{21}} \right) \right) \ln f_2 dv d\eta_{l_2}\\ +
 \int \int \frac{T_{12}}{m_1} \nabla_{v} \cdot \left( \widetilde{M}_{12} \nabla_v \left( \frac{M_{1}}{\widetilde{M}_{12}}  \right) \right)\ln M_1 dv d\eta_{l_1} \\+ \int \int \frac{T_{12}}{m_1} \nabla_{\eta_{l_1}} \cdot \left( \widetilde{M}_{12} \nabla_{\eta_{l_1}} \left( \frac{M_{1}}{\widetilde{M}_{12}} \right) \right) \ln M_1 dv d\eta_{l_1} \\
 +
 \int \int \frac{T_{21}}{m_2} \nabla_{v} \cdot \left( \widetilde{M}_{21} \nabla_v \left( \frac{M_{2}}{\widetilde{M}_{21}}  \right) \right)\ln M_2 dv d\eta_{l_2} \\+ \int \int \frac{T_{21}}{m_2} \nabla_{\eta_{l_2}} \cdot \left( \widetilde{M}_{21} \nabla_{\eta_{l_2}} \left( \frac{M_{2}}{\widetilde{M}_{21}} \right) \right) \ln M_2 dv d\eta_{l_2}\leq 0\end{align*}
with equality if and only $f_1$ and $f_2$ are Maxwell distributions with equal mean velocities $u=u_1=u_2$ and temperatures $T=T_1^{r}=T_2^{r}=T_1^{t}=T_2^{t}=\Lambda_1=\Lambda_2=\Theta_1=\Theta_2$.
\label{H-theorem}
\end{theorem}
\begin{remark}
The inequality in the H-Theorem is still true if $l_1=0$ or $l_2=0$ which means that one species is mono atomic. In this case only the equalities with $\Theta_1$ and $\Theta_2$, respectively in the local equilibrium vanish.
\end{remark}
\begin{proof}
The term $I$ can be written in the equivalent form
\begin{align*} 
I = \int \int  \left( \nabla_v   \cdot ((v-u_{12})f_1) + \nabla_v \cdot ( \frac{\Lambda_{12}}{m_1} \nabla_v f_1) \right)\ln f_1 dv d\eta_{l_1} \\+ \int \int  \left( \nabla_v   \cdot ((v-u_{21})f_2) + \nabla_v \cdot \left( \frac{\Lambda_{21}}{m_2} \nabla_v f_2 \right) \right)\ln f_2 dv d\eta_{l_2} 
\\ +  \int \int  \left( \nabla_{\eta_{l_1}}   \cdot (\eta_{l_1} f_1) + \nabla_{\eta_{l_1}} \cdot \left( \frac{\Theta_{12}}{m_1} \nabla_{\eta_{l_1}} f_1 \right) \right)\ln f_1 dv d\eta_{l_1}
\\ +  \int \int  \left( \nabla_{\eta_{l_2}}   \cdot (\eta_{l_2} f_2) + \nabla_{\eta_{l_2}} \cdot \left( \frac{\Theta_{21}}{m_2} \nabla_{\eta_{l_2}} f_2 \right) \right)\ln f_2 dv d\eta_{l_2}
\\  + z \int \int  \left( \nabla_v   \cdot ((v-u_{12}) M_1) + \nabla_v \cdot \left( \frac{T_{12}}{m_1} \nabla_v M_1 \right) \right)\ln M_1 dv d\eta_{l_1}
\\+ z \int \int  \left( \nabla_v   \cdot ((v-u_{21}) M_2) + \nabla_v \cdot \left( \frac{T_{21}}{m_2} \nabla_v M_2 \right) \right)\ln M_2 dv d\eta_{l_2}\\ +z  \int \int  \left( \nabla_{\eta_{l_1}}   \cdot (\eta_{l_1} M_1) + \nabla_{\eta_{l_1}} \cdot \left( \frac{T_{12}}{m_1} \nabla_{\eta_{l_1}} M_1 \right) \right)\ln M_1 dv d\eta_{l_1} 
\\ + z \int \int  \left( \nabla_{\eta_{l_2}}   \cdot (\eta_{l_2} M_2) + \nabla_{\eta_{l_2}} \cdot \left( \frac{T_{21}}{m_2} \nabla_{\eta_{l_2}}M_2 \right) \right)\ln M_2 dv d\eta_{l_2} 
\end{align*}

After integration by parts, we obtain
\begin{align*} 
I =-  \int \int  \sum_{i=1}^d \left( (v_i-u_{12,i}) \partial_{v_i} f_1 + \frac{\Lambda_{12}}{m_1}  \frac{(\partial_{v_i} f_1)^2}{f_1}  \right) dv d\eta_{l_1}
\\ -  \int \int  \sum_{i=1}^d \left( (v_i-u_{21,i}) \partial_{v_i} f_2 + \frac{\Lambda_{21}}{m_2}  \frac{(\partial_{v_i} f_2)^2}{f_2}  \right) dv d\eta_{l_2}
\\ -  \int \int  \sum_{i=1}^d \left( \eta_{l_{1,i}} \partial_{\eta_{l_{1,i}}} f_1 + \frac{\Theta_{12}}{m_1}  \frac{(\partial_{\eta_{\eta_{l_1,i}}} f_1)^2}{f_1}  \right)  dv d\eta_{l_1}
\\ -  \int \int  \sum_{i=1}^d \left( \eta_{l_{2,i}} \partial_{\eta_{l_{2,i}}} f_2 + \frac{\Theta_{21}}{m_2}  \frac{(\partial_{\eta_{\eta_{l_2,i}}} f_2)^2}{f_2}  \right)  dv d\eta_{l_2}
\\- z \int \int  \sum_{i=1}^d \left( (v_i-u_{12,i}) \partial_{v_i} M_1 + \frac{T_{12}}{m_1} \frac{ \left( \partial_{v_i} M_1 \right)^2}{M_1}  \right)  dv d\eta_{l_1}
\\- z \int \int  \sum_{i=1}^d \left( (v_i-u_{21,i}) \partial_{v_i} M_2 + \frac{T_{21}}{m_2}  \frac{\left(\partial_{v_i} M_2 \right)^2}{M_2}  \right)  dv d\eta_{l_1}\\ - z \int \int   \sum_{i=1}^d \left( \eta_{l_{1,i}} \partial_{\eta_{l_{1,i}}} M_1 + \frac{T_{12}}{m_1}  \frac{\left(\partial_{\eta_{\eta_{l_1,i}}} M_1 \right)^2}{M_1}  \right)  dv d\eta_{l_1}
\\ -  z \int \int   \sum_{i=1}^d \left( \eta_{l_{2,i}}  \partial_{\eta_{l_{2,i}}} M_2 + \frac{T_{21}}{m_2}  \frac{ \left(\partial_{\eta_{\eta_{l_2,i}}} M_2 \right)^2}{M_2}  \right)  dv d\eta_{l_2}
\\=: I_{1}+I_{2}+ I_{3}+I_{4} +I_5 +I_6 +I_7 +I_8
 \end{align*}
We start with $I_1+I_2+I_3+I_4$. We handle this sum similar as  in the one species case. Note that, for instance in $I_1$, now we add and subtract the term $\frac{v_i - u_{12,i}}{\Lambda_{12}}$ instead of $\frac{v_i-u_{1i}}{\Lambda_1}$. Therefore, we have to compute integrals over $|v-u_{12} |^2 f_1$ instead of $|v-u_1|^2 f_1$, which leads to the additional appearance of velocity terms. We obtain
\begin{align*}
I_1 + I_2 +I_3 +I_4 \leq \varepsilon c_{21} d n_1 n_2 (\frac{T_1^t + m_1 |u_1-u_{12}|^2}{\Lambda_{12}} -1) + \varepsilon c_{21} l_1 n_1 n_2 (\frac{T_1^r }{\Theta_{12}} -1) \\ +  c_{21} d n_1 n_2 (\frac{T_2^t + m_2 |u_2-u_{21}|^2}{\Lambda_{21}} -1) + c_{21} l_2 n_1 n_2 (\frac{T_2^r }{\Theta_{21}} -1) \\ = \varepsilon c_{21} d n_1 n_2 (\frac{T_1^t + m_1 (1-\delta)^2 |u_1-u_{2}|^2}{\Lambda_{12}} -1) + \varepsilon c_{21} l_1 n_1 n_2 (\frac{T_1^r }{\Theta_{12}} -1) \\ +  c_{21} d n_1 n_2 (\frac{T_2^t + m_2 (\frac{m_1}{m_2})^2 \varepsilon^2 (1-\delta)^2 |u_2-u_{1}|^2}{\Lambda_{21}} -1) + c_{21} l_2 n_1 n_2 (\frac{T_2^r }{\Theta_{21}} -1)
\\ = \varepsilon c_{21} d n_1 n_2 (\frac{T_1^t + m_1 (1-\delta)^2 |u_1-u_{2}|^2- \Lambda_{12}}{\Lambda_{12}} ) + \varepsilon c_{21} l_1 n_1 n_2 (\frac{T_1^r -\Theta_{12}}{\Theta_{12}} ) \\ +  c_{21} d n_1 n_2 (\frac{T_2^t + m_2 (\frac{m_1}{m_2})^2 \varepsilon^2 (1-\delta)^2 |u_2-u_{1}|^2- \Lambda_{21}}{\Lambda_{21}} ) + c_{21} l_2 n_1 n_2 (\frac{T_2^r- \Theta_{21} }{\Theta_{21}} )
\end{align*}
In the second equality, we used the expressions of $u_{12},u_{21}$ given by \eqref{veloc}.
Simliar, the other four terms can be estimated by
\begin{align*}
I_5+I_6&+I_7+I_8 \\ &\leq \varepsilon z c_{21} d n_1 n_2 (\frac{\Lambda_1 + m_1 (1-\delta)^2 |u_1-u_{2}|^2- T_{12}}{T_{12}} ) + \varepsilon z c_{21} l_1 n_1 n_2 (\frac{\Theta_1 -T_{12}}{T_{12}} ) \\ &+  c_{21} d z n_1 n_2 (\frac{\Lambda_2 + m_2 (\frac{m_1}{m_2})^2 \varepsilon^2 (1-\delta)^2 |u_2-u_{1}|^2- T_{21}}{T_{21}} ) + c_{21} z l_2 n_1 n_2 (\frac{\Theta_2- T_{21} }{T_{21}} )
\end{align*}
We start with $I_1+I_2+I_3+I_4$. We use the conservation of internal energy \eqref{internal} , the second assumption in assumptions \ref{assump} together with the definition of $T_{12}, T_{21}$ $(\Rightarrow T_{12}=\Lambda_{12}=\Theta_{12}; T_{21}=\Lambda_{21}=\Theta_{21})$ to get
\begin{align*}
I_1+I_2+I_3+I_4   \leq \varepsilon c_{21}  n_1 n_2 (\frac{d \Lambda_1 + m_1 (1-\delta)^2 |u_1-u_{2}|^2 +l_1 \Theta_1- (d \Lambda_{12} +l_1 \Theta_{12})}{T_{12}} )  \\ +  c_{21}  n_1 n_2 (\frac{d \Lambda_2 + m_2 (\frac{m_1}{m_2})^2 \varepsilon^2 (1-\delta)^2 |u_2-u_{1}|^2 +l_2 \Theta_2- (d \Lambda_{21}+l_2 \Theta_{21})}{T_{21}} ) 
\end{align*}
We insert the expressions of $d \Lambda_{12} +l_1 \Theta_{12},$ $d\Lambda_{21} +l_2 \Theta_{21}$ given by \eqref{contemp}, \eqref{temp} and obtain
\begin{align*}
I_1+I_2+I_3+I_4 \leq  c_{21} n_1 n_2 (\varepsilon \frac{(1-\alpha)(d \Lambda_1 +l_1 \Theta_1 -(d \Lambda_2 +l_2 \Theta_2)) + \tilde{\gamma}_1 |u_1-u_2|^2}{T_{12}}\\ +  \frac{\varepsilon(1-\alpha)(d \Lambda_2 +l_2 \Theta_2 -(d \Lambda_1 +l_1 \Theta_1)) + \tilde{\gamma}_2 |u_1-u_2|^2}{T_{21}}\end{align*}
with $\tilde{\gamma}_1= m_1 (1-\delta)^2-\gamma$ and $\tilde{\gamma}_2= m_2 (1-\delta)^2 \varepsilon^2 (m_1/m_2)^2-\varepsilon m_1(1-\delta) + \varepsilon \gamma$. Under the assumptions \ref{assump} on $\gamma$ and $\delta$, we can estimate $\tilde{\gamma}_1$ and $\tilde{\gamma}_2$ by zero from above
\begin{align}
\begin{split}
I_1+I_2+I_3+I_4\leq c_{21} n_1 n_2 \varepsilon (1-\alpha)(d \Lambda_1 +l_1\Theta_2)-(d \Lambda_2 +l_2\Theta_2))(\frac{1}{T_{12}} -\frac{1}{T_{21}}) \\= c_{21} n_1 n_2 \varepsilon (1-\alpha)(d \Lambda_1 +l_1\Theta_1)-(d \Lambda_2 +l_2 \Theta_2))\frac{T_{21}-T_{12}}{T_{12}T_{21}} 
\label{I1-I4}
\end{split}
\end{align}
We can compute \begin{align}
\begin{split}
T_{21}-T_{12}&= (1-\varepsilon(1-\alpha)) T_2 + \varepsilon (1-\alpha) \frac{d+l_1}{d+l_2} T_1 - \alpha T_1 -(1-\alpha) \frac{d+l_2}{d-l_1} T_2 \\ &= (\varepsilon(1-\alpha) \frac{d+l_1}{d+l_2} -\alpha) T_1 +(1 -\varepsilon (1-\alpha)\frac{d+l_2}{d+l_1}) T_2 \\&= (\alpha -\varepsilon(1-\alpha) \frac{d+l_1}{d+l_2}) ( \frac{d+l_2}{d+l_1} T_2 -T_1) \\ &=  (\alpha -\varepsilon(1-\alpha) \frac{d+l_1}{d+l_2}) \frac{1}{d+l_1}(d \Lambda_2 +l_2 \Theta_2 -(d\Lambda_1 +l_1 \Theta_1))
\label{T12-T21}
\end{split}
\end{align}
Note that due to the assumption on $\gamma$ in assumptions \ref{assump} the velocity terms do not show up. 
Under the assumptions \ref{assump} on $\alpha$ this leads to a positive sign of $(\alpha -\varepsilon(1-\alpha) \frac{d+l_1}{d+l_2})$. Therefore if we insert \eqref{T12-T21} into \eqref{I1-I4}, we obtain $I_1+I_2+I_3+I_4 \leq 0$.
In the same way, we can estimate the terms $I_5+I_6+I_7+I_8$. If we check all inequalities in this proof, we observe that we have equality  if and only $f_1$ and $f_2$ are Maxwell distributions with equal mean velocities $u=u_1=u_2$ and temperatures $T=T_1^{r}=T_2^{r}=T_1^{t}=T_2^{t}=\Lambda_1=\Lambda_2=\Theta_1=\Theta_2$.\end{proof}
 Using the definition \eqref{H-funct}, we can compute 
\begin{align*} \partial_t H(f_1,f_2) = I_1+I_2 +I
\end{align*}
 by multiplying the Fokker-Planck equation for species $1$ by $\ln f_1$, the Fokker-Planck equation for the species $2$ by $\ln f_2$, equations \eqref{kin_Temp} by $ z \ln M_k$ and sum the integrals with respect to $v$ and $\eta_{l_1}$ and $\eta_{l_2}$, respectively. This leads to the following corollary. 
 \begin{cor}[Entropy inequality for mixtures]
Under the same assumption as in theorem \ref{H-theorem} we have the following entropy inequality
\begin{align*}
\partial_t \left(  H(f_1,f_2) \right)  \leq 0,
\end{align*}
with equality if and only $f_1$ and $f_2$ are Maxwell distributions with equal mean velocities $u=u_1=u_2$ and temperatures $T=T_1^{r}=T_2^{r}=T_1^{t}=T_2^{t}=\Lambda_1=\Lambda_2=\Theta_1=\Theta_2$.
\end{cor}
\section{A mixture consisting of a mono and a polyatomic gas}
\label{sec9}
We consider now the special case of two species, one species is mono-atomic and has only translational degrees of freedom $l_1=0$, the other one has $l_2$ rotational degrees of freedom. Both have the number of degrees of freedom in translations given by $d$ with $d\in \mathbb{N}$. In this case the total number of rotational degrees of freedom is $M=l_1+l_2=l_2$. Our variables for the rotational energy degrees of freedom are $\eta \in \mathbb{R}^2$, $ \eta_{l_1} = \begin{pmatrix}
0 \\ 0 
\end{pmatrix},$ $ \eta_{l_2}= \eta $, since $\eta_{l_k}$ coincides with $\eta$ in the components corresponding to the rotational degrees of freedom of species $k$ and is zero in the other components. 
So our distribution function $f_1(x,v,t)$ of species $1$ depends on $x,v,$ and $t$ and our distribution function $f_2(x,v,\eta,t)$ of species $2$ depends on $x,v, \eta$ and $t$. The moments of $f_1$ are given by 
\begin{align}
\int f_1(v) \begin{pmatrix}
1 \\ v  \\ m_1 |v-u_1|^2  
\end{pmatrix} 
dv =: \begin{pmatrix}
n_1 \\ n_1 u_1  \\ d n_1 T_1^{t} 
\end{pmatrix} 
\end{align} 
and the moments of species $2$ are given by 
\begin{align}
\int f_2(v, \eta) \begin{pmatrix}
1 \\ v \\ \eta \\ m_2 |v-u_2|^2 \\ m_2 |\eta  |^2 \end{pmatrix} 
dvd\eta=: \begin{pmatrix}
n_2 \\ n_2 u_2 \\ 0 \\ d n_2 T_2^{t} \\ l_2 n_2 T_k^{r} 
\end{pmatrix} .
\end{align} 
The third equality is an assumption. We could also consider a general $\bar{\eta}$.
Our model reduces to 
\begin{align} \begin{split} 
\partial_t f_1   &= c_{11} n_1 \left( \frac{\Lambda_{11}}{m_1} \nabla_{v} \cdot \left( M_1 \nabla_{v} \left( \frac{f_1}{M_1} \right) \right) 
 \right)
+ c_{12} n_2 \left( \frac{\Lambda_{12}}{m_1} \nabla_{v} \cdot \left( M_{12} \nabla_{v} \left( \frac{f_1}{M_{12}} \right) \right)  
\right)\\ 
\partial_t f_2  &=c_{22} n_2 \left( \frac{\Lambda_{22}}{m_2} \nabla_{v} \cdot \left( M_2 \nabla_{v} \left( \frac{f_2}{M_2} \right) \right) + \frac{\Theta_{22}}{m_2} \nabla_{\eta} \cdot \left( M_2 \nabla_{\eta} \left( \frac{f_2}{M_{2}} \right) \right)  \right)
\\&+ c_{21} n_1 \left( \frac{\Lambda_{21}}{m_2} \nabla_{v} \cdot \left( M_{21} \nabla_{v} \left( \frac{f_2}{M_{21}} \right) \right)  + \frac{\Theta_{21}}{m_2} \nabla_{\eta} \cdot \left( M_{21} \nabla_{\eta} \left( \frac{f_2}{M_{21}} \right) \right)\right)
\end{split}
\end{align}
with the modified Maxwellians
{\footnotesize
\begin{align} 
\begin{split}
M_1(f_1)(x,v,t) &= \frac{n_1}{\sqrt{2 \pi \frac{T_1^t}{m_1}}^d }   \exp\left(- \frac{1}{2} \frac{m_1 |v-u_1|^2}{T_1^t}\right), 
\\
M_2(f_2)(x,v,\eta,t) &= \frac{n_2}{\sqrt{2 \pi \frac{\Lambda_2}{m_2}} } \frac{1}{\sqrt{2 \pi \frac{\Theta_2}{m_2}}^{l_2}}  \exp\left(- \frac{1}{2} \frac{m_2 |v-u_2|^2}{\Lambda_2}- \frac{1}{2} \frac{m_2 |\eta|^2}{\Theta_2}\right), 
\\
M_{12}(x,v,t) &= \frac{n_{12}}{\sqrt{2 \pi \frac{\Lambda_{12}}{m_1}}^d }  \exp \left({- \frac{|v-u_{12}|^2}{2 \frac{\Lambda_{12}}{m_1}}}\right),
\\
M_{21}(x,v,\eta,t) &= \frac{n_{21}}{\sqrt{2 \pi \frac{\Lambda_{21}}{m_2}}^d } \frac{1}{\sqrt{2 \pi \frac{\Theta_{21}}{m_2}}^{l_2}} \exp \left({- \frac{|v-u_{21}|^2}{2 \frac{\Lambda_{21}}{m_2}}}- \frac{|\eta|^2}{2 \frac{\Theta_{21}}{m_2}}\right),
\end{split}
\end{align}}
 For $\Lambda_2$ we use the additional relaxation equation 
 \begin{align}
\partial_t M_2  = \frac{c_{22} n_2}{Z_2^r} \frac{d+l_2}{d} (\widetilde{M}_2 - M_2) + c_{21} n_1 (\widetilde{M}_{21} - M_2), 
\label{kin_TempESapp}
\end{align}
Here, $\widetilde{M}_2$ is given by
\begin{align}
\widetilde{M}_2 = \frac{n_2}{\sqrt{2\pi \frac{T_2}{m_2}}} \frac{1}{\sqrt{2 \pi \frac{T_2}{m_2}}^{l_2}} \exp \left( - \frac{1}{2}\frac{m_2 |v-u_2|^2}{T_2} - \frac{1}{2}  \frac{m_2 |\eta|^2 }{ T_2} \right),
\end{align}
where  $T_2$ is defined as
\begin{align}
\begin{split}
T_2 := \frac{d}{d+2} \Lambda_2 + \frac{2}{d+2} \Theta_2 
\end{split}
\label{Tenapp}
\end{align}
 We couple this with conservation of internal energy of species $2$
 \begin{align}
\frac{d}{2} n_2 \Lambda_2 = \frac{d}{2} n_2 T_2^{t} +\frac{l_2}{2} n_2 T_2^{r} - \frac{l_2}{2} n_2 \Theta_2. 
\end{align} 
If we multiply \eqref{kin_TempESapp} by $|\eta|^2$ and integrate with respect to $v$ and $\eta$, this leads to the following macroscopic equation 
\begin{align}
\begin{split}
\partial_t ( \Lambda_2)  = \frac{c_{22} n_2}{Z_r^2} \frac{d+2}{d} ( T_2 - \Lambda_2) + c_{21} n_1 (T_{21} - \Theta_2).
\end{split}
\end{align}
If we assume that \begin{align*} n_{12}=n_1 \quad \text{and} \quad n_{21}=n_2, \\ u_{12}= \delta u_1 + (1- \delta) u_2, \quad \delta \in \mathbb{R},
\end{align*}
and
\begin{align}
\begin{split}
d \Lambda_{12} &=  \alpha T_1^{t} + ( 1 - \alpha) (d \Lambda_2 +l_2 \Theta_2) + \gamma |u_1 - u_2 | ^2,  \quad 0 \leq \alpha \leq 1, \gamma \geq 0, 
\label{contemp3}
\end{split}
\end{align}

we have conservation of mass, total momentum and total energy provided that
\begin{align}
u_{21}=u_2 - \frac{m_1}{m_2} \varepsilon (1- \delta ) (u_2 - u_1),
\end{align}
\begin{align}
\begin{split}
d \Lambda_{21} + l_2 \Theta_{21}=\left[  \varepsilon m_1 (1- \delta)  - \varepsilon \gamma \right] |u_1 - u_2|^2 \\+ \varepsilon ( 1 - \alpha ) d T_1^{t} + ( 1- \varepsilon ( 1 - \alpha)) (d  \Lambda_2 + l_2  \Theta_2).
\end{split}
\end{align}
If we make the assumptions \eqref{assump} the H-Theorem applies with $l_1=0$, $\Lambda_1=T_1^t$ . Note that the  assumption $\Lambda_{12}=\Theta_{12}$ is not necessary, since there is no $\Theta_{12}$ and \eqref{contemp3} already provides an explicit formula for $\Lambda_{12}$. 


%
%
%

{\bf Acknowledgements}

Marlies Pirner was funded by the Deutsche Forschungsgemeinschaft (DFG, German Research Foundation) under Germany’s Excellence Strategy EXC 2044-390685587, Mathematics Münster: Dynamics–Geometry–Structure,  and the German Science Foundation DFG (grant no. PI 1501/2-1).

%
%

\noindent

\bigskip
\bibliographystyle{abbrv}
\bibliography{sn-bibliography}

\end{document}